\documentclass[12pt, reqno]{amsart}
\usepackage{amsmath, amsthm, amscd, amsfonts, amssymb, graphicx, color}
\usepackage[bookmarksnumbered, colorlinks, plainpages]{hyperref}
\hypersetup{colorlinks=true,linkcolor=red, anchorcolor=green, citecolor=cyan, urlcolor=red, filecolor=magenta, pdftoolbar=true}

\usepackage{amsfonts}
\usepackage{color,graphicx,shortvrb}
\usepackage{enumerate}
\usepackage{amsmath}
\usepackage{amssymb}
\usepackage{graphicx}

\usepackage{enumerate}
\usepackage{amsmath}
\usepackage{amssymb}

\usepackage[latin 1]{inputenc}

\usepackage{amsmath}

\newtheorem{theorem}{Theorem}[section]

\newtheorem{corollary}[theorem]{Corollary}
\newtheorem{lemma}[theorem]{Lemma}

\newtheorem{problem}[theorem]{Problem}

\newtheorem{remark}[theorem]{Remark}

\usepackage{centernot}

\usepackage{tikz}
\usetikzlibrary{arrows,matrix}

\begin{document}

\title{Weak-2-local symmetric maps on C$^*$-algebras}

\date{\today}

\author[J.C. Cabello]{Juan Carlos Cabello}
\address{Departamento de An{\'a}lisis Matem{\'a}tico, Universidad de Granada,\\
Facultad de Ciencias 18071, Granada, Spain}

\email{jcabello@ugr.es}

\author[A.M. Peralta]{Antonio M. Peralta}
\address{Departamento de An{\'a}lisis Matem{\'a}tico, Universidad de Granada,\\
Facultad de Ciencias 18071, Granada, Spain}

\email{aperalta@ugr.es}

\thanks{Authors partially supported by the Spanish Ministry of Economy and Competitiveness project no. MTM2014-58984-P and Junta de Andaluc\'{\i}a grants FQM375. The second author acknowledges the support from the Deanship of Scientific Research, King Saud University (Saudi Arabia) research group no. RG-1435-020.}

\keywords{weak-2-local symmetric maps; weak-2-local $^*$-derivations; weak-2-local $^*$-homomorphisms}

\subjclass[2010]{47B49, 46L05, 46L40, 46T20, 47L99}

\date{October 2nd, 2015
\newline \indent $^{*}$ Corresponding author}

\begin{abstract} We introduce and study weak-2-local symmetric maps between C$^*$-algebras $A$ and $B$ as non necessarily linear nor continuous maps $\Delta: A\to B$ such that for each $a,b\in A$ and $\phi\in B^{*}$, there exists a symmetric linear map $T_{a,b,\phi}: A\to B$, depending on $a$, $b$ and $\phi$, satisfying $\phi \Delta(a) = \phi T_{a,b,\phi}(a)$ and $\phi \Delta(b) = \phi T_{a,b,\phi}(b)$. We prove that every weak-2-local symmetric map between C$^*$-algebras is a linear map. Among the consequences we show that every weak-2-local $^*$-derivation on a general C$^*$-algebra is a (linear) $^*$-derivation. We also establish a 2-local version of the Kowalski-S{\l}odkowski theorem for general C$^*$-algebras by proving that every 2-local $^*$-homomorphism between C$^*$-algebras is a (linear) $^*$-homomorphism.
\end{abstract}

\maketitle

\section{Introduction}

Recent papers, like \cite{KOPR2014} and \cite{EssaPeRa14}, show how the notions of local derivations and homomorphisms (cf. \cite{LarSou,Kad90,John01}) can be studied in terms of the local reflexivity of the sets of derivations and homomorphisms in the sense employed in \cite{BattMol}. Let us briefly recall some basic definitions. Let $\mathcal{S}$ be a subset of the space $L(X,Y)$ of all linear maps between Banach spaces $X$ and $Y$. A linear mapping $\Delta : X\to Y$ is said to be a \emph{local $\mathcal{S}$ map} if for each $x\in X$, there exists $T_{x}\in \mathcal{S}$, depending on $x$, satisfying $\Delta(x) = T_{x}(x)$. When $\mathcal{S}$ is the set $\mathcal{D}(A,X)$ of all (bounded) derivations from a C$^*$-algebra $A$ into a Banach $A$-bimodule $X$, local $\mathcal{D}(A,X)$ maps are called \emph{local derivations}. Local ($^*$-)homomorphisms, local ($^*$-)automorphisms, local Jordan ($^*$-)homomorphisms and ($^*$-)automorphisms are similarly defined. Some remarkable results on local maps read as follow: B.E. Johnson proved in \cite{John01} that every local derivation from a C$^*$-algebra $A$ into a Banach $A$ bimodule is a derivation. Every local triple derivation on a JB$^*$-triple is a triple derivation (see \cite{BurFerPe2013}). For an infinite dimensional separable Hilbert space $H$, every local automorphism on the Banach algebra $B(H)$ is an automorphism (compare \cite{LarSou,BreSemrl95}). The Gleason-Kahane-\.{Z}elazko theorem (cf. \cite{Gle,KaZe}) asserts that every unital linear local homomorphism from a unital complex Banach algebra $A$ into $\mathbb{C}$ is multiplicative.\smallskip

If in the definition of local $\mathcal{S}$ maps, we relax the linearity assumptions on the mapping $\Delta$ and we require a ``good'' local behavior at two points we discover the notion of 2-local maps.  A (non-necessarily linear nor continuos) mapping $\Delta : X\to Y$ is said to be a \emph{2-local $\mathcal{S}$ map} if for each $x,y\in X$, there exists $T_{x,y}\in \mathcal{S}$, depending on $x$ and $y$, satisfying $\Delta(x) = T_{x,y}(x)$ and $\Delta(y) = T_{x,y}(y)$. Some achievements in this line include the following:

\begin{itemize}
\item The Kowalski-S{\l}odkowski theorem (cf. \cite{KoSlod}) proves that every 2-local homomorphism $T$ from a (not necessarily commutative nor unital) complex Banach algebra $A$ into $\mathbb{C}$ is linear and multiplicative. Consequently, every (not necessarily linear) 2-local homomorphism $T$ from $A$ into a commutative C$^*$-algebra is linear and multiplicative;

\item P. \v{S}emrl established in \cite{Semrl97} that, for every infinite-dimensional separable Hilbert space $H$, every 2-local automorphism (respectively, every 2-local derivation) on $B(H)$ is an automorphism (respectively, a derivation). Sh. Ayupov and K. Kudaybergenov proved that \v{S}emrl's theorem also holds for arbitrary Hilbert spaces \cite{AyuKuday2012}. See \cite{KimKim04,Fos2012,Fos2014} and  \cite{Mol2003} for other related results;

\item Every linear 2-local $^*$-homomorphism (respectively, every bounded and linear 2-local homomorphism) between C$^*$-algebras is a homomorphism (cf. \cite{Pe2015});

\item A Kowalski-S{\l}odkowski theorem type has been recently established by M. Burgos, F.J. Fernández-Polo, J. Garcés and the second author of this note in \cite{BuFerGarPe2015RACSAM}. The main result in the just quoted paper proves that every (not necessarily linear) 2-local $^*$-homomorphism from a von Neumann algebra or from a compact C$^*$-algebra into another C$^*$-algebra is linear and a $^*$-homomorphism;

\item The same authors mentioned in the previous point show in \cite{BuFerGarPe2015JMAA} that every 2-local triple homomorphism from a JBW$^*$-triple into a JB$^*$-triple is linear and a triple homomorphism;

\item Sh. Ayupov, K. Kudaybergenov and I. Rakhimov recently prove that every 2-local (Lie) derivation on a finite-dimensional semi-simple Lie algebra over an algebraically closed field of characteristic zero is a derivation (cf. \cite{AyuKudRakh2015});

\item K. Kudaybergenov, T.~Oikhberg, B.~Russo and the second author of this paper prove in  \cite{KOPR2014} that every 2-local triple derivation on a von Neumann algebra is linear and a triple derivation.
\end{itemize}

Weaker versions of local and 2-local maps have been recently explored in \cite{EssaNiPe14,EssaPeRa14, NiPe2014} and \cite{NiPe2015}. A linear mapping $\Delta : X\to Y$ is said to be a \emph{weak local $\mathcal{S}$ map} if for each $x\in X$ and a functional $\phi$ in the dual, $Y^*$, of $Y$, there exists $T_{x,\phi}\in \mathcal{S}$, depending on $x$ and $\phi$, satisfying $\phi \Delta(x) = \phi T_{x,\phi}(x)$. One of the main results in \cite{EssaPeRa14} establishes that every weak-local derivation on a C$^*$-algebra is a derivation. It is also proved in \cite[\S 4]{EssaPeRa14} that every weak-local $^*$-automorphism on $C(\Omega)$ is multiplicative and a $^*$-homomorphism.\smallskip

The problem whether every weak-local $^*$-automorphism on a general C$^*$-algebra is multiplicative remains open until now. Under certain additional hypothesis, some partial answer can be addressed. Let $A$ be a C$^*$-algebra. A linear mapping $\Delta : A\to A$ is said to be a \emph{strong-local $^*$-automorphism} if for each $a\in A$ and each positive functional $\phi$ in $A^*$, there exists and $^*$-automorphism $\pi_{a,\phi} : A\to A$, depending on $a$ and $\phi$, satisfying $$\|\Delta(a) -\pi_{a,\phi}(a)\|_{\phi}^2 = \phi \left((\Delta(a) -\pi_{a,\phi}(a))^* (\Delta(a) -\pi_{a,\phi}(a)) \right)=0.$$ Motivated by results due to L. Moln\'{a}r (see \cite{Mol2014}), Theorem 4.1 in \cite{EssaPeRa14} establishes that every strong-local $^*$-automorphism on a von Neumann algebra is a Jordan $^*$-homomorphism, where in this case, the definition of strong-local $^*$-automorphism is exactly the same with $\phi$ being a normal positive functional. According to our knowledge, the following problems remain open:

\begin{problem}\label{problem weak-local automorphisms} Is every (linear) weak local $^*$-automorphism on a C$^*$-algebra a Jordan $^*$-homomorphism?
\end{problem}

\begin{problem}\label{problem Kowalski-Slodkowski C*algebras} Is every 2-local $^*$-homomorphism on a C$^*$-algebra a linear $^*$-homomorphism.? {\rm(}compare \cite{BuFerGarPe2015RACSAM}{\rm)}.
\end{problem}

The weak version of 2-local maps has been recently explored in \cite{NiPe2014} and \cite{NiPe2015}. We recall that a (non-necessarily linear nor continuos) mapping $\Delta : X\to Y$ is said to be a \emph{weak-2-local $\mathcal{S}$ map} if for each $x,y\in X$ and $\phi\in Y^{*}$, there exists $T_{x,y,\phi}\in \mathcal{S}$, depending on $x$, $y$ and $\phi$, satisfying $\phi \Delta(x) = \phi T_{x,y,\phi}(x)$ and $\phi \Delta(y) = \phi T_{x,y,\phi}(y)$.\smallskip

It is proved in \cite{NiPe2014,NiPe2015} that every weak-2-local derivation on $M_n$, the algebra of $n$ times $n$ square matrices, is a linear derivation. Actually, every weak-2-local derivation on a finite dimensional C$^*$-algebra is a linear derivation. Further, the main result in \cite{NiPe2015} shows that, for every separable complex Hilbert space $H$, every weak-2-local $^*$-derivation (i.e. a derivation which is also a symmetric mapping) on $B(H)$ is a linear $^*$-derivation. In view of the just quoted results and references it seems natural to consider the following problems:

\begin{problem}\label{problem weak-local derivations linear} Is every weak-2-local $^*$-derivation on a C$^*$-algebra linear? Is every weak-2-local $^*$-derivation on a C$^*$-algebra a derivation?
\end{problem}

\begin{problem}\label{problem weak-2-local automorphisms linear} Is every weak-2-local $^*$-homomorphism on a C$^*$-algebra linear? Is every weak-2-local $^*$-homomorphism on a C$^*$-algebra a $^*$-homomorphism?
\end{problem}

Let $\Delta :A \rightarrow B$ be a mapping between $C^*$-algebras, and let $A_{sa}$ and $B_{sa}$ denote the self-adjoint parts of $A$ and $B$, respectively. We define a new mapping $\Delta^\sharp: A \rightarrow B$ given by $\Delta^\sharp(a):=\Delta (a^*)^*$ ($a \in A$). We say that $\Delta$ is \emph{symmetric} if $\Delta^{\sharp} = \Delta$ (equivalently, $\Delta (a^*)= \Delta(a)^*$, for all $a\in A$). Clearly, $\Delta^{\sharp \sharp}=\Delta$, $\Delta(A_{sa}) \subseteq B_{sa}$ whenever $\Delta$ is symmetric, and $\Delta$ is linear (respectively, a derivation or a homomorphism) if and only if $\Delta^\sharp$ is linear (respectively, a derivation or a homomorphism). Henceforth, the symbol $\mathcal{S}(A,B)$ will stand for the set of all linear and symmetric maps from $A$ into $B$. We write $\mathcal{S}(A)$ for $S(A,A)$.  In this paper we shall study the general class of \emph{weak-2-local $\mathcal{S}(A,B)$-maps} or \emph{weak-2-local symmetric maps} between C$^*$-algebras $A$ and $B$.\smallskip

The main result in this note shows that every weak-2-local symmetric map between C$^*$-algebras is linear (Theorem \ref{thm weak-2-local symmetric maps are linear}). Consequently, every weak-2-local $^*$-derivation on a C$^*$-algebra and every weak-2-local $^*$-homomorphism between C$^*$-algebras is automatically linear (Corollary \ref{c weak-2-local starderivations and starhom are linear}). These results provide positive answers to the first questions in Problems \ref{problem weak-local derivations linear} and \ref{problem weak-2-local automorphisms linear}, and a partial solution to the conjecture posed at the end of the introduction of \cite{BuFerGarPe2015RACSAM} (see Problem \ref{problem Kowalski-Slodkowski C*algebras}). In this case, a more general point of view makes easier the arguments.\smallskip

Concerning Problem \ref{problem Kowalski-Slodkowski C*algebras}, our conclusions go further. Namely, in Corollary \ref{c 2-local starhoms are starhoms} we establish a 2-local version of the Kowalski-S{\l}odkowski theorem for general C$^*$-algebras by proving that every 2-local $^*$-homomorphism between C$^*$-algebras is a linear $^*$-homomorphism. This provides a positive answer to the just quoted problem and to the question posed in the introduction of \cite{BuFerGarPe2015RACSAM}.\smallskip

Our main result also throws some new light to the study of 2-local derivations on general C$^*$-algebras. In \cite{AyuKuday2014} Sh. Ayupov and K. Kudaybergenov culminated the study of 2-local derivations on von Neumann algebras by proving that each 2-local derivation on an arbitrary von Neumann algebra is a derivation (see \cite{AyuKudPe2014} for more details). Corollary \ref{c 2-local starderivations are linear} below proves that every 2-local $^*$-derivation on a C$^*$-algebra is a linear $^*$-derivation, while Corollary \ref{c w-2-local starderivations are linear starderivations} shows that the same conclusion remains true for weak-2-local $^*$-derivations on general C$^*$-algebras. It should be remarked here that, by Corollary 2.13 in \cite{NiPe2015} every weak-2-local derivation on a finite dimensional C$^*$-algebra is a linear derivation. The question whether the above conclusion remains valid for arbitrary weak-2-local derivations on general C$^*$-algebras remains open. \smallskip

The techniques and arguments presented in this note are completely new and independent from the ideas in previous forerunners. Actually, the new point of view introduced in this note by considering weak-2-local symmetric maps allows us to provide simpler proofs.\smallskip

\subsection{Notation and background}

A functional $\phi$ in the dual of a C$^*$-algebra $A$ is symmetric if and only if $\phi (A_{sa})\subseteq \mathbb{R}$. The set of all symmetric functionals in $A^*$ will be denoted by $(A^*)_{sa}$.  It is known that the mapping $(A^*)_{sa} \rightarrow (A_{sa})^*,$ $ \phi \mapsto \Re\hbox{e} \phi,$ is a surjective linear isometry. We shall indistinctly write $A^{*}_{sa}$ for $(A^*)_{sa}$ and for $(A_{sa})^*$. Clearly, for each $\phi \in A_{sa}^*$ we have $\phi (a^*) = \overline{\phi(a)}$, for all $a\in A$. It is also known that, for each $a$ in $A_{sa}$ there exists $\phi\in A^*_{sa}$ satisfying $\|\phi\|=1$ and $\phi (a) = \|a\|$ (compare \cite[Proposition 1.5.4]{S}). Consequently, $A^{*}_{sa}$ separates the points of $A$. \smallskip

\section{Weak-2-local symmetric maps}

This section is completely devoted to the study of weak-2-local symmetric maps on a C$^*$-algebra. We begin with a lemma containing the basic results on weak-2-local maps.

\begin{lemma}\label{l basic properties 1} Let $X$ and $Y$ be Banach spaces and let $\mathcal{S}$ be a subset of the space $L(X,Y)$. Then the following properties hold:\begin{enumerate}[$(a)$] \item Every weak-2-local $\mathcal{S}$ map $\Delta: X\to Y$ is 1-homogeneous, that is, $\Delta (\lambda x) =  \lambda \Delta(x)$, for every $x\in X$, $\lambda\in \mathbb{C}$;
\item  Suppose there exists $C> 0$ such that every linear map $T\in \mathcal{S}$ is continuous with $\|T\|\leq C$. Then every weak-2-local $\mathcal{S}$ map $\Delta: X\to Y$ is $C$-Lipschitzian, that is, $\|\Delta(x)-\Delta (y) \|\leq C \|x-y\|$, for every $x,y\in X$;
\item If $\mathcal{S}$ is a (real) linear subspace of $L(X,Y)$, then every (real) linear combination of weak-2-local $\mathcal{S}$ maps is a weak-2-local $\mathcal{S}$ map;
\item Suppose $A$ and $B$ are C$^*$-algebras and $\mathcal{S}$ is a real linear subspace of $L(A,B)$. Then a mapping $\Delta :A \rightarrow B$ is a weak-2-local $\mathcal{S}$ map if and only if for each $\varphi\in B_{sa}^*$ and every $x,y\in A$, there exists $T_{x,y,\varphi}\in \mathcal{S}$ satisfying $\varphi \Delta(x) = \varphi T_{x,y,\varphi}(x)$ and $\varphi \Delta(y) = \varphi T_{x,y,\varphi}(y)$.
\item Suppose $A$ and $B$ are C$^*$-algebras and $\mathcal{S}$ is a real linear subspace of $L(A,B)$ with $\mathcal{S}^{\sharp} = \mathcal{S}$ (in particular when $\mathcal{S} = \mathcal{S}(A,B)$ is the set of all symmetric linear maps from $A$ into $B$). Then a mapping $\Delta :A \rightarrow B$ is a weak-2-local $\mathcal{S}$ map if and only if $\Delta^\sharp$ is a weak-2-local $\mathcal{S}$ map.
\end{enumerate}
\end{lemma}

\begin{proof} Part of the proof follows by arguments similar to those in \cite[Lemma 2.1.(b)]{NiPe2014} and relies on the Hahn-Banach theorem. We include an sketch of the proof here for completeness reasons.\smallskip

$(a)$ Let $\phi$ be an arbitrary functional in $Y^{*}$, let $x$ be an element in $X$ and $\lambda$ a complex number. By assumptions, for each weak-2-local $\mathcal{S}$ map $\Delta : X\to Y$ there exists $T_{x,\lambda x,\phi} \in \mathcal{S}$ satisfying $\phi \Delta (x) = \phi T_{x,\lambda x, \phi} (x)$ and $\phi \Delta (\lambda x) = \phi T_{x,\lambda x,\phi} (\lambda x)$. Since $\mathcal{S}\subset L(X,Y)$, we deduce that $\phi (\lambda \Delta (x)) = \phi \Delta (\lambda x)$, for every $\phi \in Y^*$. The Hahn-Banach theorem gives the desired statement.\smallskip

$(b)$ Pick an arbitrary norm-one functional $\phi\in Y^*$. Given $x,y\in X$, we have $$|\phi (\Delta (x) -  \Delta (y))| = |\phi (T_{x,y,\phi} (x) - T_{x,y,\phi} (x))| \leq C \|x-y\|.$$ Taking supremum in $\phi \in Y^*$ with $\|\phi\|\leq 1$, we get $\|\Delta (x) -  \Delta (y)\|\leq  C \|x-y\|$.\smallskip

$(c)$ is clear.\smallskip

$(d)$  Suppose $\Delta: A\to B$ is a map. If $\Delta$ is a weak-2-local $\mathcal{S}$ map, then for each $\varphi\in B_{sa}^*$ and every $x,y\in A$, there exists $T_{x,y,\varphi}\in \mathcal{S}$ satisfying $\varphi \Delta(x) = \varphi T_{x,y,\varphi}(x)$ and $\varphi \Delta(y) = \varphi T_{x,y,\varphi}(y)$. Suppose that $\Delta$ satisfies the last property. Every $\phi \in B^*$ writes in the form $\phi =\varphi_1+i \varphi_2$, where $\varphi_j\in B^*_{sa}$. So, given $x,y\in A$, there exist $T_{x,y,\varphi_1},T_{x,y,\varphi_2}\in \mathcal{S}$ satisfying  $\varphi_j \Delta(x) = \varphi_j T_{x,y,\varphi_j}(x)$ and $\varphi_j \Delta(y) = \varphi T_{x,y,\varphi_j}(y)$, for every $j=1,2$. Then $\phi \Delta (x) = \phi (T_{x,y,\varphi_1}+ T_{x,y,\varphi_2}) (x) $ and $\phi \Delta (y) = \phi (T_{x,y,\varphi_1}+ T_{x,y,\varphi_2}) (y),$ which shows that $\Delta$ is a weak-2-local $\mathcal{S}$ map.\smallskip

$(e)$ Suppose $\Delta: A\to B$ is a weak-2-local $\mathcal{S}$ map. Let us take $\phi\in B^*_{sa}$, $x,y$ in $A$. Since $\phi \Delta^{\sharp} (x) = \phi T^{\sharp}_{x^*,y^*,\phi} (x)$ and $\phi \Delta^{\sharp} (y) = \phi T^{\sharp}_{x^*,y^*,\phi} (y)$. We conclude from $(d)$ that $\Delta^{\sharp}$ is a weak-2-local $\mathcal{S}$ map. The rest is clear.
\end{proof}

\begin{remark}\label{r star homomorphisms are contractive} Let us recall that every $^*$-homomorphism between C$^*$-algebras is contractive. So, it follows from the above Lemma \ref{l basic properties 1}$(b)$ that every weak-2-local $^*$-homomorphism {\rm(}and every weak-2-local $^*$-automorphism{\rm)} between C$^*$-algebras is 1-Lipschitzian.
\end{remark}

\begin{remark}\label{r symmetric maps are a subspace} Let $\mathcal{S} (A,B)$ be the set of all symmetric linear maps from a C$^*$-algebra $A$ into a C$^*$-algebra $B$. Clearly $\mathcal{S} (A,B)$ is a real subspace of $L(A,B)$ with $\mathcal{S} (A,B)^{\sharp} = \mathcal{S} (A,B)$. Obviously $\mathcal{S} (A,B)$ contains all $^*$-homomorphisms and all $^*$-derivations on $A$. Therefore, the statements in Lemma \ref{l basic properties 1} $(c)$-$(e)$ apply for weak-2-local symmetric maps between C$^*$-algebras.
\end{remark}

\begin{lemma}\label{l w2l symmetric maps are symmetric} Let $\Delta: A\to B$ be a weak-2-local symmetric map between C$^*$-algebras. Then the following statements hold:\begin{enumerate}[$(a)$]\item $\Delta (a^*) = \Delta(a)^*$, for every $a\in A$. In particular,  $\Delta (A_{sa}) \subseteq B_{sa}$;
\item $\Delta(h+ik)=\Delta(h)+i\Delta(k)=\Delta(h-ik)^*$, for every
$h,k \in A_{sa}$;
\item $\Delta(a+a^*)=\Delta(a)+\Delta(a^*) $, for every $a\in A$.
\end{enumerate}
\end{lemma}

\begin{proof}
$(a)$ Let us fix $\phi\in B^{*}_{sa}$ and $a\in A$. We can write $\Delta(a) = h_1 + i k_1$, $\Delta(a^*) = h_2 + i k_2,$ where $h_j,k_j\in B_{sa}$. By the weak-2-local property of $\Delta$, there exists a symmetric linear map $T_{a,a^*,\phi}: A\to B$ satisfying $$\phi (h_1) + i \phi (k_1) = \phi \Delta(a) = \phi T_{a,a^*,\phi} (a),$$ and $$\phi (h_2) + i \phi (k_2) = \phi \Delta(a^*) = \phi T_{a,a^*,\phi} (a^*) = \overline{\phi T_{a,a^*,\phi} (a)}.$$
The above identities show that $\phi (h_2)= \phi (h_1)$ and $\phi (k_2)= -\phi (k_1)$, for every $\phi \in B^{*}_{sa}$. Therefore $h_1=h_2$ and $k_1=-k_2$, which proves that $\Delta(a^*) = h_2 + i k_2 = h_1-i k_1= \Delta(a)^{*}$.\smallskip

$(b)$ Let us fix $h,k \in A_{sa}$ and $\phi \in B^*_{sa}$. By hypothesis there exists a symmetric linear map $T_{h,h+ik,\phi}: A\to B$, depending on $h,h+ik$ and
$\phi$, such that
$$ \phi \Delta (h+ik)=\phi T_{h,h+ik,\phi}(h+ik)=\phi T_{h,h+ik,\phi}(h)+ i
\phi T_{h,h+ik,\phi}(k)$$ and  $$\phi \Delta(h)= \phi
T_{h,h+ik,\phi}(h),$$ with $\phi T_{h,h+ik,\phi}(k), \phi T_{h,h+ik,\phi}(h) \in \mathbb{R}.$  Then $\Re\hbox{e} \phi \Delta(h+ik)= \phi \Delta(h)$ for every  $\phi \in B^*_{sa}$, which implies that $ \Delta(h+ik)+\Delta(h+ik)^* = 2\Delta(h)$.\smallskip

By Lemma \ref{l basic properties 1}$(a)$, and the arguments given above, we have $$-i (\Delta(h+ik) -\Delta(h+ik)^*)=  -i \Delta(h+ik)+ i \Delta(h+ik)^* $$ $$= \Delta(-i h+ k) + \Delta(-i h+ k)^* = 2 \Delta (k).$$ Thus, $\Delta(h+ik)=\Delta(h)+\Delta(i k) =\Delta(h)+i \Delta( k)  =\Delta(h-ik)^*.$\smallskip

$(c)$ Let us take $a\in A$ and write $a= h+ i k$ with $h,k\in A_{sa}$. Lemma \ref{l basic properties 1}$(a)$ gives $\Delta (a+a^*) = \Delta (2 h) = 2 \Delta(h)$. On the other hand, by statement $(b)$, $\Delta (a) + \Delta (a^*) = \Delta(a) + \Delta (a)^* = \Delta(h) + i \Delta(k) +\Delta(h) - i \Delta(k) =   2 \Delta(h)$.
\end{proof}

The main result of the paper can be stated now.

\begin{theorem}\label{thm weak-2-local symmetric maps are linear}
Every weak-2-local symmetric map between $C^*$-algebras is linear.
\end{theorem}

\begin{proof}
Let $\Delta:A \rightarrow B$ be a weak-2-local symmetric map between $C^*$-algebras. Applying Lemma \ref{l w2l symmetric maps are symmetric}$(b)$ we deduce that, in order to prove that $\Delta$ is linear, it is enough to show that $\Delta(h+k)=\Delta(h)+\Delta(k),$ for every $h,k \in A_{sa}.$\smallskip

Let us fix $h, k \in A_{sa}$, and write $a=h+ik$. We observe that $a+ia^*=(1+i)(h+k) $, and so
\begin{equation}\label{41} \Delta(a+ia^*)=(1+i)\Delta(h+k),\end{equation} where $\Delta(h+k) \in B_{sa}$
(compare Lemmas \ref{l basic properties 1} and \ref{l w2l symmetric maps are symmetric}).\smallskip

We deduce from Lemma \ref{l w2l symmetric maps are symmetric}$(a)-(b)$ that \begin{equation}\label{42} \Delta(a)+i \Delta(a)^*= (1+i) (\Delta(h)+\Delta(k)).\end{equation}

Let us fix an arbitrary functional $\phi \in B^*_{sa}$. By hypothesis there are symmetric linear maps $T_{a+ i a^*, a, \phi}, T_{a+ i a^*, a^*, \phi} : A\to B$ satisfying: $$\phi \Delta (a+ia^*)= \phi T_{a+ i a^*, a, \phi} (a+ia^*)=\phi T_{a+ i a^*, a^*, \phi} (a+ia^*)$$ $$\phi \Delta (a)=
\phi T_{a+ i a^*, a, \phi} (a)  \hbox{ and } \phi \Delta (a^*) = \phi T_{a+ i a^*, a^*, \phi}(a^*)= \phi T_{a+ i a^*, a^*, \phi}(a)^*.$$

In particular, $$\phi (T_{a+ i a^*, a^*, \phi}(a)^* + T_{a+ i a^*, a, \phi}(a) )= \phi \Delta (a^*) + \phi
\Delta (a) $$ $$=\hbox{(by Lemma \ref{l w2l symmetric maps are symmetric}$(a)-(c)$)}= \phi \Delta (a^*+a) \in \mathbb{ R},$$
and since $\phi (T_{a+ i a^*, a^*, \phi} (a) + T_{a+ i a^*, a^*, \phi} (a)^*) \in \mathbb{ R}$, we deduce that $$\phi (T_{a+ i a^*, a^*, \phi}-T_{a+ i a^*, a, \phi})(a) \in \mathbb{R},$$ and so \begin{equation} \label{50} \phi (T_{a+ i a^*, a^*, \phi}-\Delta)(a) \in \mathbb{R}.\end{equation}

On the other hand,
\begin{equation}\label{eq 1 0608} \phi \Delta(a+ia^*)= \phi T_{a+ i a^*, a, \phi} (a+ia^*)=\phi T_{a+ i a^*, a, \phi} (a)+i \phi T_{a+ i a^*, a, \phi}(a^*)
\end{equation} $$= \phi \Delta (a) + i \phi T_{a+ i a^*, a, \phi}(a^*)$$
and
\begin{equation}\label{eq 2 0608} \phi \Delta (a+ia^*)= \phi T_{a+ i a^*, a^*, \phi}(a+ia^*)=\phi T_{a+ i a^*, a^*, \phi}(a)+i \phi T_{a+ i a^*, a^*, \phi}(a^*)
\end{equation} $$= \phi T_{a+ i a^*, a^*, \phi}(a) + i \phi \Delta (a^*).$$

Therefore, \begin{equation} \label{51} \phi \Delta (a)+i \phi T_{a+ i a^*, a, \phi} (a^*)= \phi T_{a+ i a^*, a^*, \phi}(a) + i \phi \Delta(a^*).\end{equation}\smallskip

Combining  \eqref{50} and \eqref{51} we get
$$\phi (\Delta (ia^*)-T_{a+ i a^*, a, \phi}(ia^*)) 
= \phi \Delta (a)-\phi T_{a+ i a^*, a^*, \phi}(a) \in \mathbb{ R},$$ which implies that \begin{equation}\label{eq 3 0608} \Im\hbox{m} \phi \Delta(ia^*)= \Im\hbox{m} \phi T_{a+ i a^*, a, \phi}(ia^*).
\end{equation}

Clearly, by \eqref{eq 1 0608} and \eqref{eq 2 0608}, the identity $$ 2 \phi \Delta (a+ia^*)=  \phi \Delta (a) + \phi T_{a+ i a^*, a^*, \phi} (a)  +   \phi T_{a+ i a^*, a, \phi}  (i a^*)+ \phi \Delta (i a^*),$$ holds. Since, from \eqref{50}, $\Im\hbox{m} \phi \Delta(a)=\Im\hbox{m} \phi T_{a+ i a^*, a^*, \phi}(a)$,
it follows from the last identity and \eqref{eq 3 0608} that $$2\Im\hbox{m} \phi \Delta (a+ia^*)= 2 \Im\hbox{m} \phi \Delta(a) +  2 \Im\hbox{m} \phi \Delta (i a^*),$$ or equivalently, $$\Im\hbox{m} \phi \left(\Delta (a+ia^*)- \Delta(a) - \Delta (i a^*)\right)=0.$$ The arbitrariness of $\phi\in B_{sa}^*$, implies that $\Delta (a+ia^*)- \Delta(a) -\Delta (i a^*) \in A_{sa}.$ Thus, by \eqref{41} and \eqref{42} $$(1+i)[\Delta(h+k)-\Delta(h)-\Delta(k)]=\Delta(a+ia^*)-(\Delta(a)+i\Delta(a)^*) \in A_{sa}.$$ Finally, since Lemma \ref{l w2l symmetric maps are symmetric}$(a)$ assures that  $\Delta(h+k),\Delta(h), \Delta(k) \in B_{sa}$, we deduce that $$0=\Delta(h+k)-\Delta(h)-\Delta(k),$$ which completes the proof.
\end{proof}

\begin{corollary}\label{c weak-2-local starderivations and starhom are linear} Every weak-2-local $^*$-derivation on a C$^*$-algebra and every weak-2-local $^*$-homomorphism between C$^*$-algebras is linear.$\hfill\Box$
\end{corollary}

We recall that 2-local $^*$-derivations and 2-local $^*$-homomorphisms are weak-2-local $^*$-derivations and weak-2-local $^*$-homomorphisms, respectively.

\begin{corollary}\label{c 2-local starhoms are linear} Every 2-local $^*$-homomorphism between C$^*$-algebras is linear.$\hfill\Box$
\end{corollary}

The main result in \cite{BuFerGarPe2015RACSAM} proves that every 2-local $^*$-homomorphism from a von Neumann algebra into a C$^*$-algebra is linear and a $^*$-homomorphism. The main difficulty in the proof of this result is to show that every such a mapping is linear. It is also remarked at the introduction of \cite{BuFerGarPe2015RACSAM}, that ``it would be of great interest to explore if the same conclusion remains true for general C$^*$-algebras''. Surprisingly, we can give now a very simple proof of the general result as a consequence of our previous corollary  and the results in \cite{Pe2015}. More concretely, Theorem 3.9 in \cite{Pe2015} combined with the above Corollary \ref{c 2-local starhoms are linear} establish that every  2-local $^*$-homomorphism on a C$^*$-algebra is a $^*$-homomorphism. We can give now a positive answer to Problem \ref{problem Kowalski-Slodkowski C*algebras}.

\begin{corollary}\label{c 2-local starhoms are starhoms} Every 2-local $^*$-homomorphism between C$^*$-algebras is a linear $^*$-homomorphism.$\hfill\Box$
\end{corollary}

As we have already commented, the previous corollary provides a generalized Kowalski-S{\l}odkowski theorem for C$^*$-algebras in the sense of \cite{BuFerGarPe2015RACSAM}. It is also, in some sense, an extension of \cite[Theorem 4.1]{EssaPeRa14}. We do not know if every (linear) weak-2-local $^*$-homomorphism on a C$^*$-algebra is a $^*$-homomorphism. This question remains as an open problem.\smallskip

We deal now with 2-local $^*$-derivations.

\begin{corollary}\label{c 2-local starderivations are linear} Every 2-local $^*$-derivation on a C$^*$-algebra is a linear $^*$-derivation.
\end{corollary}

\begin{proof} Let $\Delta : A \to A$ be a 2-local $^*$-derivation on a C$^*$-algebra. Theorem \ref{thm weak-2-local symmetric maps are linear} shows that $\Delta$ is linear. Further, given $a\in A$, there exists a $^*$-derivation $D_{a,a^2}: A\to A$ such that $$\Delta (a) = D_{a,a^2} (a) \hbox{ and } \Delta (a^2) = D_{a,a^2} (a^2).$$ Therefore $$\Delta (a^2) = D_{a,a^2} (a^2) =  D_{a,a^2} (a) a + a D_{a,a^2} (a) = \Delta (a) a + a \Delta (a),$$ for every $a\in A$, which shows that $\Delta$ is a Jordan derivation on $A$. Finally, it is known, by Johnson's theorem (compare \cite{John96}, see also \cite[Corolaries 17 and 18]{PeRu}), that Jordan derivations on a C$^*$-algebra are derivations.
\end{proof}

We shall see next that the conclusion of the above Corollary \ref{c 2-local starderivations are linear} can be improved.

\begin{corollary}\label{c w-2-local starderivations are linear starderivations} Every weak-2-local $^*$-derivation on a C$^*$-algebra is a linear $^*$-derivation.
\end{corollary}

\begin{proof} Let $\Delta : A\to A$ be a weak-2-local $^*$-derivation on a C$^*$-algebra. Theorem \ref{thm weak-2-local symmetric maps are linear} implies that $\Delta$ is a linear mapping. Since every weak-2-local $^*$-derivation on a C$^*$-algebra is a weak-local $^*$-derivation, Theorem 3.4 in \cite{EssaPeRa14} assures that $\Delta$ is a derivation.
\end{proof}

Sh. Ayupov and K. Kudaybergenov have recently studied 2-local derivations on von Neumann algebras in \cite{AyuKuday2014}. The main conclusion established in the just quoted paper shows that every 2-local derivation on an arbitrary von Neumann algebra is a derivation. Our previous Corollary extends the study of Ayupov and Kudaybergenov to the class of weak-2-local $^*$-derivations on general C$^*$-algebras.


\begin{thebibliography}{99}

\bibitem{AyuKuday2012} Sh. Ayupov, K. Kudaybergenov, 2-local derivations and automorphisms on $B(H)$, \emph{J. Math. Anal. Appl.} \textbf{395}, no. 1, 15-18 (2012).

\bibitem{AyuKuday2014} Sh. Ayupov, K.K. Kudaybergenov, $2$-local derivations on von Neumann algebras, \emph{Positivity} \textbf{19}, No. 3, 445-455 (2015). DOI 10.1007/s11117-014-0307-3

\bibitem{AyuKudPe2014} Sh. Ayupov, K. Kudaybergenov, A.M. Peralta, A survey on local and 2-local derivations on C$^*$- and von Neuman algebras, to appear in \emph{Conteporary Math. Amer. Math. Soc.} arXiv:1411.2711v1.

\bibitem{AyuKudRakh2015}  Sh. Ayupov, K. Kudaybergenov, I. Rakhimov, 2-local derivations on finite-dimensional Lie algebras, \emph{Linear Algebra Appl.} \textbf{474}, 1-11. (2015).

\bibitem{BattMol} C. Batty, L. Molnar, On topological reflexivity of the groups of *-automorphisms and surjective isometries of $B(H)$, \emph{Arch. Math.} \textbf{67} 415-421 (1996).

\bibitem{BreSemrl95} M. Bre\v{s}ar, P. \v{S}emrl, On local automorphisms and mappings that preserve idempotents, \emph{Studia Math.} \textbf{113}, no. 2, 101-108  (1995).

\bibitem{BuFerGarPe2015RACSAM} M. Burgos, F.J.  Fern{\'a}ndez-Polo, J.J.  Garc{\'e}s and A.M. Peralta, A Kowalski-Slodkowski theorem for 2-local $*$-homomorphisms on von Neumann algebras, \emph{Rev. R. Acad. Cienc. Exactas Fís. Nat. Ser. A Math. RACSAM} \textbf{109}, no. 2, 551-568 (2015).

\bibitem{BuFerGarPe2015JMAA} M. Burgos, F.J. Fern{\' a}ndez-Polo, J.J.
Garc{\'e}s, A.M. Peralta, 2-local triple homomorphisms on von Neumann algebras and JBW$^*$-triples, \emph{J. Math. Anal. Appl.} \textbf{426}, 43-63 (2015).

\bibitem{BurFerPe2013} M. Burgos, F.J. Fern{\' a}ndez-Polo, A.M. Peralta, Local triple derivations on C$^*$-algebras and JB$^*$-triples, \emph{Bull. London Math. Soc.} \textbf{46} (4), 709-724 (2014). doi:10.1112/blms/bdu024.

\bibitem{EssaNiPe14} A.B.A. Essaleh, M. Niazi, A.M. Peralta, Bilocal $^*$-automorphisms of $B(H)$ satisfying the 3-local property, \emph{Arch. Math.} \textbf{104} (2015), 157-164.

\bibitem{EssaPeRa14} A.B.A. Essaleh, A.M. Peralta, M.I. Ramírez, Weak-local derivations and homomorphisms on C$^*$-algebras, to appear in \emph{Linear and Multilinear Algebra}. doi:10.1080/03081087.2015.1028320

\bibitem{Fos2012} A. Fo\v{s}ner, 2-local Jordan automorphisms on operator algebras, \emph{Studia Math.} \textbf{209}, no. 3, 235-246 (2012).

\bibitem{Fos2014} A. Fo\v{s}ner, 2-local mappings on algebras with involution, \emph{Ann. Funct. Anal.} \textbf{5}, no. 1, 63-69 (2014).

\bibitem{Gle} A.M. Gleason, A characterization of maximal ideals, \emph{J. Analyse Math.} \textbf{19}, 171-172 (1967).

\bibitem{John96} B.E. Johnson, Symmetric amenability and the nonexistence of Lie and
Jordan derivations, \emph{Math. Proc. Cambridge Philos. Soc.}
\textbf{120}, no. 3, 455-473 (1996).

\bibitem{John01} B.E. Johnson, Local derivations on C$^*$-algebras are derivations,
\emph{Trans. Amer. Math. Soc.} \textbf{353}, 313-325 (2001).

\bibitem{Kad90} R.V. Kadison, Local derivations, \emph{J. Algebra} \textbf{130}, 494-509 (1990).

\bibitem{KaZe} J.P. Kahane, W. \.{Z}elazko, A characterization of maximal ideals in commutative Banach algebras,
\emph{Studia Math.} \textbf{29}, 339-343 (1968).

\bibitem{KimKim04} S.O. Kim, J.S. Kim, Local automorphisms and derivations on $\mathbb{M}_n$, \emph{Proc. Amer. Math. Soc.} \textbf{132}, no. 5, 1389-1392 (2004).

\bibitem{KoSlod} S. Kowalski, Z. S{\l}odkowski, A characterization of multiplicative linear functionals in
Banach algebras, \emph{Studia Math.} \textbf{67}, 215-223 (1980).


\bibitem{KOPR2014}
{K.~K.~Kudaybergenov, T.~Oikhberg, A.~M.~Peralta, B.~Russo,}
\textit{2-Local triple derivations on von Neumann algebras,} to appear in \emph{Illionis J. Math.}
arXiv:1407.3878.

\bibitem{LarSou} D.R. Larson and A.R. Sourour, Local derivations and local automorphisms of $B(X)$, \emph{Proc. Sympos. Pure Math.} \textbf{51}, Part 2, Providence, Rhode Island 1990, pp. 187-194.

\bibitem{Mol2003} L. Molnar, Local automorphisms of operator algebras on Banach spaces, \emph{Proc. Amer. Math. Soc.} \textbf{131}, 1867-1874 (2003).

\bibitem{Mol2014} L. Moln\'{a}r, Bilocal $^*$-automorphisms of $B(H)$ \emph{Arch. Math.} \textbf{102}, 83-89 (2014).

\bibitem{NiPe2014} M. Niazi and A.M. Peralta, Weak-2-local derivations on $\mathbb{M}_n$, to appear in \emph{FILOMAT}.

\bibitem{NiPe2015} M. Niazi and A.M. Peralta, Weak-2-local $^*$-derivations on $B(H)$ are linear $^*$-derivations, \emph{Linear Algebra Appl.} \textbf{487}, 276-300 (2015).

\bibitem{Pe2015} A.M. Peralta, A note on 2-local representations of C$^*$-algebras, \emph{Operators and Matrices} \textbf{9}, Number 2, 343-358 (2015). doi:10.7153/oam-09-20.

\bibitem{PeRu} A. M. Peralta and B. Russo, Automatic continuity of triple derivations on
C$^*$-algebras and JB$^*$-triples, \emph{J. Algebra} \textbf{399}, 960-977 (2014).

\bibitem{S} S. Sakai: $C^*$-algebras and $W^*$-algebras. Springer Verlag. Berlin (1971).

\bibitem{Semrl97}  P. \v{S}emrl, Local automorphisms and derivations on $B(H)$, \emph{Proc. Amer. Math. Soc.} \textbf{125}, 2677-2680 (1997).




\end{thebibliography}
\end{document}